\documentclass[11pt]{amsart}

\usepackage[margin=2cm]{geometry}
\usepackage{palatino}
\usepackage{latexsym, dsfont, amssymb, amsfonts, amsmath, amsthm, graphics, graphicx, subfigure, bbm,tikz, mathtools}
\usepackage{pgfplots}
\pgfplotsset{compat=1.18}
\usepackage[pdftex,colorlinks,linkcolor=blue,menucolor=red,backref=false,bookmarks=true]{hyperref}

\usepackage{wrapfig, graphicx}
\graphicspath{ {images/} }
\DeclareGraphicsExtensions{.png,.pdf}
\usepackage[T1]{fontenc}

\newcommand{\field}[1]{\mathbb{#1}}
\newcommand{\R}{\field{R}}

\newcommand{\E}{\field{E}}

\newcommand{\p}{\field{P}}

\newcommand{\Area}{\mathrm{Area}}
\newtheorem{theorem}{Theorem}

\newtheorem*{convention}{Convention}

\newtheorem{proposition}{Proposition}

\newtheorem{definition}{Definition}
\newtheorem{lemma}{Lemma}

\newtheorem{remark}{Remark}

\newtheorem{example}{Example}

\numberwithin{equation}{section}

\begin{document}
\title{A property of log-concave and weakly-symmetric distributions for two step approximations of random variables}

\author{Mihaela-Adriana Nistor}
\address{University of Bucharest\
  Faculty of Mathematics and Computer Science\
  Str. Academiei nr.14, sector 1, C.P. 010014, Bucharest, Romania}
\email[Mihaela-Adriana Nistor]{mihaelaadriana.nistor@gmail.com}

\author{Ionel Popescu}
\address{University of Bucharest\
  Faculty of Mathematics and Computer Science\
  Str. Academiei nr.14, sector 1, C.P. 010014, Bucharest, Romania}
\address{Institute of Mathematics ``Simion Stoilow'' of the Romanian Academy\ 21 Calea Grivitei Street, 010702 Bucharest, Romania
  Bucharest, Romania}
\email[Ionel Popescu]{ionel.popescu@fmi.unibuc.ro, ionel.popescu@imar.ro}

\thanks{}
\date{}

\begin{abstract} 
In this paper we introduce a generalization of classical risk measures in which the risk is represented by a step function taking two values, corresponding to two endogenously determined market regimes. This extends the traditional framework where risk measures map random variables to single real numbers.
For the quadratic loss function, we study the optimization problem of determining the optimal regime threshold and corresponding values. In the case of log-concave distributions we give conditions for the uniqueness of the regime changing.  We treat the case of one dimension and also of multi-dimensions for elliptic distributions.  

We demonstrate the necessity of convexity through counterexamples.  

\textbf{Keywords:} Risk measures, regime-switching, log-concave distributions, convex optimization, integral inequalities

\textbf{MSC Classification:} 91G70, 60E15, 26D15, 49K30, 90C25 
\end{abstract}

\maketitle


\section{Introduction}
Risk measurement is a central topic in probability, optimization, and mathematical finance, where the loss of a financial position is modeled by a real-valued random variable and quantified via a real-valued functional. A foundational axiomatic framework for such functionals was introduced by Artzner et al.~\cite{artzner1999coherent} through the notion of coherent risk measures, characterized by monotonicity, translation invariance, positive homogeneity, and subadditivity. Subsequent work relaxed positive homogeneity and led to convex risk measures (see, e.g.,~\cite{foellmer2002convex,foellmer2016stochastic}), establishing deep connections with convex analysis and optimization (cf.~\cite{follmer2002convex,rockafellar2000optimization}). For broader discussions and comparisons of risk measures and for further perspectives (including dynamic aspects and statistical considerations), see e.g.~\cite{mcneil2015quantitative,emmer2013best,ziegel2016coherence,chun2012conditional,artzner1999application,follmer2015axiomatic,riedel2004dynamic}.

Classical risk measures such as Value-at-Risk (VaR) and Expected Shortfall (ES) compress the distribution of losses into a single number~\cite{acerbi2002expected}. While such scalar summaries are useful in practice, they do not directly reflect the presence of different market states (e.g., calm versus stressed regimes) that are commonly modeled via regime-switching dynamics~\cite{hamilton1989nonstationary,ang2002regime,hamilton1994arch,kim1999statespace} and motivate state-dependent or dynamic risk measures~\cite{acciaio2011dynamic}.

In this paper we propose and analyze a refinement of the classical setup in which risk is represented not by a single scalar but by a simple function with finitely many regimes. Concretely, given a loss function $G:\R\to[0,\infty)$ and a class of measurable functions $\widetilde{R}$, we consider the general approximation problem
\begin{equation}\label{e:0-intro}
 f^*=\underset{f\in\widetilde{R}}{\mathrm{argmin }}\, \mathbb{E}[G(X-f(X))].
\end{equation}

The class $\widetilde{R}$ we consider here consists of two-regime (two-step) functions,
\[
 f(x)=\alpha\,\mathbbm{1}_{(-\infty,t]}(x)+\beta\,\mathbbm{1}_{(t,\infty)}(x),
\]
where the regime levels $\alpha,\beta\in\R$ and the threshold $t\in\R$ are chosen optimally and depend on the law of $X$.

Our primary focus is on the loss $G(x)=x^2$ and the problem of finding an optimal two-step approximation of $X$ in an $L^2$ sense.

This formulation is closely related to optimal partitioning and quantization, where one seeks a partition and representative values minimizing an expected distortion; a central issue there is \emph{uniqueness} of the optimal partition (or of the corresponding Lloyd fixed point), which can fail without structural assumptions. For log-concave (equivalently ILR) densities, uniqueness and convergence properties are well understood for broad classes of convex losses; see Trushkin~\cite{trushkin1982sufficient,trushkin1984monotony} and the statistical treatment of unique optimal partitions by Mease and Nair~\cite{mease2006unique} (see also Eubank's survey~\cite{eubank1988optimal}). In the present paper we focus on the \emph{two-regime} case and obtain a more explicit, quantitative description: for the quadratic loss the optimization reduces to maximizing an explicit one-parameter functional $t\mapsto f_X(t)$, and under log-concavity we give conditions ensuring a unique maximizer; in the weakly-symmetric log-concave case we further identify this maximizer as the common mean/median.

Our analysis focuses on log-concave distributions, a class that plays a fundamental role in probability and analysis, and enjoys strong monotonicity and stability properties (see~\cite{SaumardWellner2014,bagnoli2005logconcave}). In this paper we take a slightly different perspective: we focus on the two-regime case, characterize when the optimal threshold is unique, and identify cases where it is also explicit (namely, the common mean/median) for weakly-symmetric log-concave laws.

\subsection*{Description of the main results}
\begin{itemize}
\item \emph{One dimension (Sections~\ref{S:2}--\ref{S:3}).}
We start from the general approximation problem~\eqref{e:0} and show that, for convex losses, once the regime levels
$\alpha,\beta$ are fixed, the optimal regime set is a half-line (Proposition~\ref{prop:optimal-A-halfline}).
For the quadratic loss this reduces the optimization to thresholds $t\in\R$ and yields the one-parameter criterion
\begin{equation*}
 f_X(t):=\frac{\mathbb{E}[X,\,X\le t]^2}{\mathbb{P}(X\le t)}+\frac{\mathbb{E}[X,\,X>t]^2}{\mathbb{P}(X>t)}
\end{equation*}
(see~\eqref{eq:pbmax1} and~\eqref{e:f:V}). Theorem~\ref{t:1} then identifies conditions ensuring that $f_X$ has a
\emph{unique maximizer} and, in the weakly-symmetric log-concave case, proves that this maximizer is exactly the
common mean/median $t=\mu$.  The treatment here is rather elementary and completely self-contained.   

\item \emph{Higher dimensions (Section~\ref{S:multi}).}
We extend the framework to $X\in\R^d$ and show that, for the quadratic loss, the optimal regime boundary is an affine
hyperplane, so it is natural to optimize over halfspaces (Proposition~\ref{prop:multi-optimal-set} and
\eqref{eq:class-H}). For centered elliptical laws with log-concave one-dimensional projections, the halfspace
optimization reduces to the one-dimensional problem along projections, and the optimal direction is characterized
via a Rayleigh quotient (Theorem~\ref{t:dim:elli}).

\item \emph{Sharpness and counterexamples (Section~\ref{SS:3:3} and Section~\ref{subsec:halfspace-ctrex}).}
We give explicit examples showing that without log-concavity/convexity the function $f_X$ can have multiple global
maximizers, and in dimension $d\ge2$ the halfspace functional need not be maximized at the centered cut even when
$\mathbb{E}[X]=0$.
\end{itemize}

\section{The one dimensional case}\label{S:2}   

In this section we introduce the main concept in the one dimensional case.  

\subsection{Problem setup}\label{SS:2:1}
Let $G:\R\to[0,\infty)$ be a loss function and consider
\begin{equation}\label{e:0}
 f^*=\underset{f\in\mathcal{\widetilde{R}}}{\mathrm{argmin }}\, \mathbb{E}[G(X-f(X))].
\end{equation}
where $\widetilde{R}$ is a class of funtions.  Our main model in our case is the one in which we have $f=\mathbbm{1}_{A}$ for some Borel measurable $A$ on the real line.  

The first result is about the structure of the set $A$ which in general under convexity assumptions on $G$ must be essentially a half line.  

\begin{proposition}\label{prop:optimal-A-halfline}
Assume that $G:\R\to\R$ is convex. Then for any fixed $\alpha,\beta\in\R$, a minimizer of
\[
A\longmapsto \E\bigl[G(X-\alpha)\,\mathbbm{1}_{A}(X)+G(X-\beta)\,\mathbbm{1}_{A^c}(X)\bigr]
\]
over Borel sets $A\subset\R$ is given (up to $\mathbb{P}$-null sets) by
\[
A_{\alpha,\beta}:=\{x\in\R:\ G(x-\alpha)\le G(x-\beta)\}.
\]
Moreover, $A_{\alpha,\beta}$ is an interval (possibly empty or all of $\R$). In particular, unless $\alpha=\beta$,
$A_{\alpha,\beta}$ is a half-line.
\end{proposition}
\begin{proof}
\emph{Step 1: pointwise minimization.}
Fix $\alpha,\beta$ and write the objective as
\[
\E\bigl[G(X-\beta)\bigr]+\E\bigl[(G(X-\alpha)-G(X-\beta))\,\mathbbm{1}_{A}(X)\bigr].
\]
Thus one minimizes by taking $\mathbbm{1}_{A}(x)=1$ exactly when $G(x-\alpha)-G(x-\beta)\le 0$, i.e. by choosing
$A=A_{\alpha,\beta}$.

\smallskip
\noindent\emph{Step 2: half-line structure under convexity.}
Assume $\alpha<\beta$ and set $\delta:=\beta-\alpha>0$. For $u\in\R$ define
$H(u):=G(u+\delta)-G(u)$. By convexity of $G$, the map $H$ is nondecreasing. Hence
$\{u:\ H(u)\le 0\}$ is an interval of the form $(-\infty,u_0]$ (or empty/all), and translating back shows that
$A_{\alpha,\beta}$ is a half-line. The case $\alpha>\beta$ is analogous, and if $\alpha=\beta$ then every $A$ is
optimal.\qedhere
\end{proof}

This result justifies the fact that we can restrict our optimization to the sets of the form $A=(-\infty,c]$ and thus the class of functions is given by 
\[
 f(x)=\alpha\,\mathbbm{1}_{(-\infty,c]}(x)+\beta\,\mathbbm{1}_{(c,\infty)}(x).
\]

In the rest of the paper we restrict to the quadratic loss $G(x)=x^2$, so the optimization becomes
\begin{equation}\label{eq:main}
    \underset{\alpha,\beta,c\in\mathbb{R}}{\text{argmin }}\mathbb{E}[(X-\alpha\mathds{1}_{(-\infty,c]}(X)-\beta\mathds{1}_{(c,\infty)}(X))^2].
\end{equation}

Before we move on, we will use the following convention. 

\begin{convention} For any event $A$, we write
\[
\E[X,A]:=\E\big[X\,\mathbbm{1}_A\big].
\]
\end{convention}

For fixed $c$ (equivalently, fixed $A$), the optimal values are the conditional means. Indeed, setting the partial
derivatives with respect to $\alpha$ and $\beta$ to zero yields
\begin{equation}\label{eq:alpha}
    \alpha=\frac{\mathbb{E}[X\mathds{1}_{(-\infty,c]}(X)]}{\mathbb{P}(X\le c)}=\mathbb{E}[X\mid X\le c],
\end{equation}
\begin{equation}\label{eq:beta}
    \beta=\frac{\mathbb{E}[X\mathds{1}_{(c,\infty)}(X)]}{\mathbb{P}(X>c)}=\mathbb{E}[X\mid X>c].
\end{equation}

Plugging these back gives the reduced objective
\[
 h(c)=\mathbb{E}[X^2]-\frac{\mathbb{E}[X, X\le c]^2}{\mathbb{P}(X\le c)}-\frac{\mathbb{E}[X, X>c]^2}{\mathbb{P}(X>c)}.
\]
Thus minimizing \eqref{eq:main} is equivalent to maximizing
\begin{equation}\label{eq:pbmax1}
    \underset{c}{\text{argmax}} \left(\frac{\mathbb{E}[X, X\le c]^2}{\mathbb{P}(X\le c)}+\frac{\mathbb{E}[X, X> c]^2}{\mathbb{P}(X>c)} \right).
\end{equation}
\vspace{0.1cm}

\subsection{Continuous and log-concave distributions in one dimension}\label{S:3}



In this section, we treat a very special case of distribution for which the structure of the optimization problem  \eqref{e:0} has unique solutions and in some particular cases is just the mean. The model we treat here is the case of random variables $X$ which have a log-concave density and are weakly-symmetric.  

\begin{definition}
By definition, we call a density $\phi:\R\to\R$ \emph{log-concave} if $\log(\phi(x))$ is concave.  This translates into $\phi(x)=e^{-V(x)}$ where $V:\R\to\R$ is a convex function.  

We say that the density $\phi$ is weakly-symmetric if the mean of $\phi$ is also a median.      
\end{definition}

In other words, if $\mu=\int\limits_{\R} x\phi(x)dx$, then $\int\limits_{(-\infty,\mu]}\phi(x)dx=1/2$.

\subsection{Examples}
If we assume that $\phi(x)=\frac{1}{Z_V}e^{-V(x)}$ with $Z_V=\int e^{-V(x)}dx$, the normalizing constant,  one such example of weakly-symmetric density is the case of symmetric density about $0$, i.e. 
\begin{equation}\label{vfunc}
    V(x)=V(-x),\ \ \forall x\in\mathbb{R}
\end{equation}
In addition, if $V$ is also convex this is an example of weakly and log-concave density.

Particular examples of weakly-symmetric are
\begin{equation*}
    \begin{cases}
        U(a,b): V(x)=\log(b-a) \text{ for }x\in[a,b],\text{ and }\infty, \text{ otherwise }\
        N(0,1): V(x)=\frac{(x-\mu)^2}{2\sigma^2}\
        Exp^2(1): V(x)=|x-\mu|
    \end{cases}
\end{equation*}

The above examples, are cases of symmetry around some real number $\mu$, i.e. they satisfy $V(x+\mu)=V(-x+\mu)$ for all $x\in\R$.  In this case, it is obvious that the median is also the mean.  

However, there are log-concave distributions which are not symmetric.  A typical example is the Weibull distribution with cumulative function given by $k\ge0$ 
\[
F(x)=e^{-x^k} \text{ for } x\ge0
\]
and for a very specific $k$ we can show that the mean can be chosen to be equal to the median.  The point is that the mean of this distribution is $\Gamma(1+1/k)$ and the median is $(\ln 2)^{1/k}$.  One can show that the mean is equal to the variance for a certain value of $k$, because for $k=1$, we have $\Gamma(1+1)=1>\ln(2)$ and for $k\to\infty$, we have $\Gamma(1+1/k)\approx 1-\gamma/k\approx1-0.577/k$, where $\gamma\approx -0.577$ is Euler's constant, and $(\ln(2))^{1/k}\approx 1-\ln(\ln(2))/k\approx 1-0.367/k$ which then means that for large $k$, $\Gamma(1+1/k)<(\ln(2))^{1/k}$.  Numerical approximation gives that $k \approx 3.439$.    

\subsection{Settings and the result.}
We consider the probability measure 
\[
\mu(dx)=\frac{1}{Z_V}e^{-V(x)}, \text{ where } Z_V=\int\limits_{\mathbb{R}} e^{-V(x)} dx.
\] 
and we consider the variable $X$ having the distribution $\mu$. Without loss of generality we assume that $Z_V=1$, so $\int\limits_{-\infty}^{\infty} e^{-V(x)} dx=1 $.

\subsection{Regularity assumptions}\label{SS:3:reg}
In Section~\ref{S:3} we assume that $X$ admits a density of the form $\phi(x)=e^{-V(x)}$ on $\R$ with $V$ convex
(so $\phi$ is log-concave) and $\int_{\R}\phi(x)\,dx=1$. We also assume $\E[X^2]<\infty$ (hence $\E[|X|]<\infty$),
so all quantities in \eqref{e:f:V} are finite.

In the proofs below we differentiate functions of the form $t\mapsto \int_{-\infty}^t h(x)\phi(x)\,dx$ and
$t\mapsto \int_t^{\infty} h(x)\phi(x)\,dx$. Since a log-concave density is locally bounded and locally integrable on the
interior of its support, these functions are absolutely continuous and their derivatives are given by $h(t)\phi(t)$ for all
$t$ in the interior of the support. When $\p(X\le t)\in\{0,1\}$, the expression \eqref{e:f:V} is interpreted by continuity and
such $t$ are irrelevant for the maximization problem.

The main result is the following.

\begin{theorem}\label{t:1}
For a random variable $X$ of mean $\mu$, set \begin{equation}\label{e:f:V}
f_{X}(t)=\frac{\mathbb{E}[X, X \le t]^{2}}{\mathbb{P}(X \le  t)}+\frac{\mathbb{E}[X, X > t]^{2}}{\mathbb{P}(X > t)}.  
\end{equation} 

\begin{enumerate}
\item If $X$ has a log-concave distribution, then $f_X$ has a unique maximizer.  

\item If $\mu$ is a local maximum for $f_X$, then $X$ must be weakly-symmetric.   

\item Assume that $X$ is a weakly-symmetric and log-concave density with mean $\mu$.  Then $f_X$ is decreasing on the interval $[\mu,\infty)$ and increasing on $(-\infty,\mu]$.  

In particular, $f(t)<f(\mu)$ for $t\ne \mu$, thus the unique optimizer of \eqref{eq:main} is $c=\mu$ and the minima obtained is 
\[
\beta(-\mathbbm{1}_{(-\infty,\mu]}(X)+\mathbbm{1}_{(\mu,\infty)}(X)) \text{ with }\beta=\E[|X-\mu|]/2.
\]
\end{enumerate}
\end{theorem}

\subsection{Some integral inequalities for convex functions}\label{SS:3:0}\hfill\

Before we provide the proof of Theorem~\ref{t:1} we introduce the key result here which is of independent interest.    

\begin{lemma}\label{lem2.1.2}
Let $V:[0,\infty)\to \mathbb{R}$ be convex and assume
\begin{equation}\label{eq:assum-int}
\int_0^\infty e^{-V(y)}\,dy<\infty,
\qquad
\int_0^\infty y\,e^{-V(y)}\,dy<\infty.
\end{equation}
Then
\begin{equation}\label{eq:goal8}
e^{-V(0)}\int_0^{\infty}y e^{-V(y)}\,dy\leq\left(\int_0^{\infty}e^{-V(y)}\,dy\right)^2.
\end{equation}
Equality is attained for $V(y)=V(0)+\lambda y$ with $\lambda>0$.  
\end{lemma}

\begin{proof}
Set, for convenience,
\[
w(y):=e^{-V(y)},\qquad y\ge 0.
\]
Since $V$ is convex, it is locally Lipschitz on $(0,\infty)$ and hence absolutely continuous on every interval $[0,R]$.
In particular, the right derivative
\[
v(y):=V'_+(y)
\]
exists for every $y\ge 0$, is nondecreasing, locally integrable, and satisfies
\[
V(y)=V(0)+\int_0^y v(s)\,ds,\qquad y\ge 0.
\]
Consequently $w$ is absolutely continuous on each $[0,R]$, and for a.e.\ $y$,
\[
w'(y)=-v(y)\,w(y).
\]

\medskip
\noindent\emph{Boundary terms needed for integration by parts.}
From $\int_0^\infty w<\infty$ we have $V(y)\to\infty$ as $y\to\infty$ (otherwise $w$ would not be integrable),
so $w(y)\to 0$. Moreover, convexity implies that $v$ is nondecreasing; hence there exists $R_0$ such that
$v(y)\ge 0$ for all $y\ge R_0$, so $V$ is nondecreasing on $[R_0,\infty)$ and thus $w$ is nonincreasing there.
For $R\ge 2R_0$ we then have
\[
\int_{R/2}^{R} w(y)\,dy \ge \frac{R}{2}\,w(R),
\qquad\text{hence}\qquad
R\, w(R)\le 2\int_{R/2}^{\infty} w(y)\,dy \xrightarrow[R\to\infty]{}0.
\]
This justifies the vanishing of the boundary term $R\,w(R)$ in the integration by parts identity below.

\medskip
\noindent\emph{Two integration-by-parts identities.}
For any $R>0$, using $w'=-v w$ a.e.\ and absolute continuity,
\[
\int_0^R v(y)w(y)\,dy=\int_0^R -w'(y)\,dy=w(0)-w(R).
\]
Letting $R\to\infty$ and using $w(R)\to 0$ gives
\begin{equation}\label{eq:ibp1}
\int_0^\infty V'_+(y)\,e^{-V(y)}\,dy = e^{-V(0)}.
\end{equation}
Similarly,
\[
\int_0^R w(y)\,dy
=\Big[y w(y)\Big]_{0}^{R}-\int_0^R y w'(y)\,dy
=R\, w(R)+\int_0^R y v(y) w(y)\,dy.
\]
Letting $R\to\infty$ and using $R\,w(R)\to 0$ yields
\begin{equation}\label{eq:ibp2}
\int_0^\infty e^{-V(y)}\,dy=\int_0^\infty y\,V'_+(y)\,e^{-V(y)}\,dy.
\end{equation}

\medskip
\noindent\emph{Proof of \eqref{eq:goal8}.}
Denote
\[
I:=\int_0^\infty e^{-V(y)}\,dy,\qquad 
J:=\int_0^\infty y\,e^{-V(y)}\,dy.
\]
Using \eqref{eq:ibp1} and \eqref{eq:ibp2}, \eqref{eq:goal8} is equivalent to
\[
\Big(\int_0^\infty y\,e^{-V(y)}\,dy\Big)\Big(\int_0^\infty V'_+(y)\,e^{-V(y)}\,dy\Big)
\le
\Big(\int_0^\infty e^{-V(y)}\,dy\Big)\Big(\int_0^\infty y\,V'_+(y)\,e^{-V(y)}\,dy\Big).
\]
Expand both sides as double integrals and symmetrize:
\begin{align*}
&\Big(\int_0^\infty e^{-V}\Big)\Big(\int_0^\infty y\,V'_+(y)\,e^{-V(y)}\,dy\Big)
-\Big(\int_0^\infty y\,e^{-V}\Big)\Big(\int_0^\infty V'_+(y)\,e^{-V(y)}\,dy\Big) \
&\hspace{2cm}
=\frac12\int_0^\infty\!\!\int_0^\infty (y_1-y_2)\bigl(V'_+(y_1)-V'_+(y_2)\bigr)\,
e^{-V(y_1)}e^{-V(y_2)}\,dy_1dy_2.
\end{align*}
Since $V'_+$ is nondecreasing, $(y_1-y_2)(V'_+(y_1)-V'_+(y_2))\ge 0$ for all $y_1,y_2$, hence
\eqref{eq:goal8} follows. Equality in \eqref{eq:goal8} forces $V'_+$ to be a.e.\ constant, i.e.\ $V$ affine:
$V(y)=V(0)+\lambda y$; integrability implies $\lambda>0$.\qedhere

\end{proof}

\subsection{The proof of Theorem~\ref{t:1}}

\begin{proof}
We structure the proof according to the three statements in the theorem.

We first reduce the proof to the case $\mu=0$.  
Let $Y:=X-\mu$. A direct check shows that for all $t\in\mathbb{R}$,
\[
 f_X(t)=f_Y(t-\mu)+\mu^2.
\]
In particular,
\[
\underset{t\in\mathbb{R}}{\arg\max} f_X(t)=\mu+\underset{s\in\mathbb{R}}{\arg\max} f_Y(s).
\]
Therefore it is enough to prove the theorem for centered variables. In the rest of the proof we assume
\begin{equation}\label{eq:centered-assumption}
\E[X]=0.
\end{equation}
and then we can write 
\begin{equation}\label{e:m0f}
    f_X(t)= \frac{\E[X,X\le t]^2}{\p(X\le t)\p(X>t).}
\end{equation}

\medskip
\begin{enumerate}
\item \emph{(Uniqueness of the maximizer under convexity of $V$.)}
Assume that $V$ is convex (equivalently, $X$ has a log-concave density $e^{-V}$) and $\E[X]=0$.
Write $f_X(t)$ in the centered form \eqref{e:m0f} and set $g(t):=\log f_X(t)$.
Differentiating gives, for all $t$ in the interior of the support,
\begin{equation}\label{e:gs}
\begin{split}
 g'(t)&=e^{-V(t)}\Big(-\frac{2t}{\int_t^\infty x e^{-V(x)}\,dx}+\frac{1}{\int_t^\infty e^{-V(x)}\,dx}-\frac{1}{\int_{-\infty}^t e^{-V(x)}\,dx}\Big)\\
 &=\frac{e^{-V(t)}}{\int_t^\infty x e^{-V(x)}\,dx}\Big(-2t+\frac{\int_t^\infty x e^{-V(x)}\,dx}{\int_t^\infty e^{-V(x)}\,dx}-\frac{\int_t^\infty x e^{-V(x)}\,dx}{\int_{-\infty}^t e^{-V(x)}\,dx}\Big).
 \end{split}
\end{equation}
We arrange this expression into  
\begin{equation}\label{e:ga}
 g'(t)=\frac{e^{-V(t)}}{\int_t^\infty x e^{-V(x)}\,dx}\left(
 \frac{\int_t^\infty (x-t) e^{-V(x)}\,dx}{\int_t^\infty e^{-V(x)}\,dx}
 -
 \frac{\int_{-\infty}^t (t-x) e^{-V(x)}\,dx}{\int_{-\infty}^t e^{-V(x)}\,dx}
 \right).
\end{equation}
Indeed, using $\E[X]=0$, we have $\int_{-\infty}^t x e^{-V(x)}\,dx=-\int_t^{\infty} x e^{-V(x)}\,dx$, hence
\[
\frac{\int_t^{\infty} x e^{-V(x)}\,dx}{\int_t^{\infty} e^{-V(x)}\,dx}
-
\frac{\int_{-\infty}^{t} x e^{-V(x)}\,dx}{\int_{-\infty}^{t} e^{-V(x)}\,dx}
-2t
=\frac{\int_t^{\infty} (x-t)e^{-V(x)}\,dx}{\int_t^{\infty} e^{-V(x)}\,dx}
-
\frac{\int_{-\infty}^{t} (t-x)e^{-V(x)}\,dx}{\int_{-\infty}^{t} e^{-V(x)}\,dx}
\]
which proves \eqref{e:ga}.  

Next define
\[
 m(t):=\frac{\int_t^{\infty} (x-t)e^{-V(x)}\,dx}{\int_t^{\infty} e^{-V(x)}\,dx},
 \qquad
 k(t):=\frac{\int_{-\infty}^{t} (t-x)e^{-V(x)}\,dx}{\int_{-\infty}^{t} e^{-V(x)}\,dx}.
\]
Thus 
\[
g'(t)=\frac{e^{-V(t)}}{\int_t^\infty x e^{-V(x)}\,dx}\left(
 m(t)-k(t) \right).  
\]
In particular, $g'(t)=0$ if and only if $m(t)=k(t)$. Moreover, since $\E[X]=0$ we have
\[
\int_t^\infty x e^{-V(x)}\,dx=-\int_{-\infty}^t x e^{-V(x)}\,dx,
\]
so the denominator is strictly positive for every $t$ in the interior of the support (for $t>0$ the first integral
is positive, while for $t<0$ the second is positive).

Consequently the sign of $g'$ is determined by $m(t)-k(t)$.  Next we analyze separately these functions.  

\smallskip
\noindent\emph{Step 1: $m$ is decreasing.}
A direct computation shows that 
\[
m'(t)=\frac{e^{-V(t)}\int_t^{\infty}(x-t)e^{-V(x)}dx - \left(\int_t^{\infty}e^{-V(x)}dx\right)^2}{\left(\int_t^{\infty}e^{-V(x)}dx\right)^2}
\]

Set $V_t(u):=V(t+u)-V(t)$, which is convex on $[0,\infty)$. Applying \eqref{eq:goal8} from
Lemma~\ref{lem2.1.2} to $V_t$ gives
\[
\int_0^{\infty} u e^{-V_t(u)}\,du\le\left(\int_0^{\infty} e^{-V_t(u)}\,du\right)^2.
\]
Rewriting back in the original variables yields
\[
 e^{-V(t)}\int_t^{\infty} (x-t)e^{-V(x)}\,dx\le\left(\int_t^{\infty} e^{-V(x)}\,dx\right)^2,
\]
which is exactly $m'(t)\le 0$, so $m$ is decreasing.

\smallskip
\noindent\emph{Step 2: $k$ is increasing.}  We can work out the same argument as for Step~1.  Alternatively, 
apply Step~1 to the log-concave random variable $-X$ (whose potential is $y\mapsto V(-y)$, still convex).
The corresponding $m$ function for $-X$ at level $-t$ is exactly $k(t)$, hence $k$ is increasing.

\smallskip
\noindent\emph{Step 3: uniqueness of the zero of $g'$.}
Since $m$ is decreasing and $k$ is increasing, $m(t)-k(t)$ is strictly decreasing.
Moreover,
\[
\lim_{t\to-\infty} m(t)=+\infty,\quad \lim_{t\to-\infty} k(t)=0,
\qquad
\lim_{t\to+\infty} m(t)=0,\quad \lim_{t\to+\infty} k(t)=+\infty,
\]
so $m(t)-k(t)$ crosses $0$ exactly once. Therefore $g'(t)$ has a unique zero, hence $g$ has a unique critical
point; since $f_X(t)\to 0$ as $t\to\pm\infty$, this critical point is the unique maximizer of $f_X$.

\item \emph{(The mean is a maximizer $\Rightarrow$ weak symmetry.)}
Assume $0$ is a maximizer of $f_X$. Then it is a critical point of $g(t)=\log f_X(t)$, hence $g'(0)=0$.  Using \eqref{e:gs} we get the conclusion that 
\[
\int_{0}^{\infty} e^{-V(x)}\,dx=\int_{-\infty}^{0} e^{-V(x)}\,dx.
\]
That is, $0$ is a median. Since $\E[X]=0$, this is exactly the weak symmetry condition.

\item \emph{(Monotonicity under weak symmetry and log-concavity.)}
Assume now that $X$ is weakly symmetric and log-concave with mean $0$. Then $0$ is a median, i.e.
\begin{equation}\label{eq:median0}
\int_{-\infty}^0 e^{-V(x)}\,dx=\int_0^{\infty} e^{-V(x)}\,dx.
\end{equation}
which combined with the expresion of \eqref{e:gs} shows that $g'(0)=0$ and also $m(0)=k(0)$.  Since we just proved that $g'$ has the same sign as $m(t)-k(t)$ which is decreasing, we get that $g'(t)<0$ for $t>0$ and $g'(t)>0$ for $t<0$ which definitely show that $g$ attains the maximum at $t=0$.  
\end{enumerate}
\qedhere
\end{proof}

\section{The multidimensional case}\label{S:multi}

In this section we extend the one–dimensional framework of Section~\ref{S:2} to random vectors
$X\in\mathbb{R}^d$.  This amounts to constructing risk strategies in the case of portfolios of risky assets.  
The guiding idea is the same: instead of approximating $X$ by a single constant
(as in classical risk measures), we allow two \emph{regimes} determined intrinsically by a measurable
set $A\subset\mathbb{R}^d$, and we approximate $X$ by two (vector) constants on $A$ and on $A^c$.

\subsection{Problem definition and classes of regimes}\label{SS:multi:prob}

Let $G:\mathbb{R}^d\to[0,\infty)$ be a given loss function and let $\widetilde{\mathcal{R}}$ be a class of
measurable functions $f:\mathbb{R}^d\to\mathbb{R}^d$. We consider the optimization problem
\begin{equation}\label{eq:multi:general}
f^*=\underset{f\in\widetilde{\mathcal{R}}}{\mathrm{argmin}}\; \mathbb{E}\bigl[\,G(X-f(X))\,\bigr].
\end{equation}
In the classical definition of risk measures, $\widetilde{\mathcal{R}}$ would be the set of constant
functions. Here we focus on a two–regime class in which $f$ takes only two vector values depending on
membership in a set $A$.

\medskip
\noindent\textbf{Standing integrability assumptions.}
For the quadratic loss $G(x)=\|x\|^2$ considered below we assume $\E\|X\|^2<\infty$ (hence $\E\|X\|<\infty$),
so that all conditional means in \eqref{eq:multi:alpha} exist and the objectives in \eqref{eq:multi:hA}--\eqref{eq:multi:pbmax}
are finite. When writing ratios such as $\frac{\|\E[X\mathbf 1_A]\|^2}{\p(A)}$ we implicitly restrict to sets with
$0<\p(A)<1$; the cases $\p(A)\in\{0,1\}$ are trivial and can be ignored in the maximization problem.

\medskip
\noindent\textbf{Standing integrability assumptions.}
For the quadratic loss $G(x)=\|x\|^2$ considered below we assume $\E\|X\|^2<\infty$ (hence $\E\|X\|<\infty$),
so that all conditional means in \eqref{eq:multi:alpha} exist and the objectives in \eqref{eq:multi:hA}--\eqref{eq:multi:pbmax}
are finite. When writing ratios such as $\frac{\|\E[X\mathbf 1_A]\|^2}{\p(A)}$ we implicitly restrict to sets with
$0<\p(A)<1$; the cases $\p(A)\in\{0,1\}$ are trivial and can be ignored in the maximization problem.

\medskip
\noindent\textbf{Two–regime class.}
For a family of measurable sets $\mathcal{A}\subset\mathcal{B}(\mathbb{R}^d)$ define
\begin{equation}\label{eq:Rgen-d}
\mathcal{R}^{(d)}_{\mathcal{A}}
=\Bigl\{\, f:\mathbb{R}^d\to\mathbb{R}^d:\ 
f(x)=\alpha\,\mathbf{1}_{A}(x)+\beta\,\mathbf{1}_{A^c}(x),\ \ A\in\mathcal{A},\ \alpha,\beta\in\mathbb{R}^d
\Bigr\}.
\end{equation}

\begin{proposition}\label{prop:multi-optimal-set}
Assume that $G:\mathbb{R}^d\to\mathbb{R}$ is convex. Fix $\alpha,\beta\in\mathbb{R}^d$.
Then a minimizer of
\[
A\longmapsto \E\bigl[G(X-\alpha)\,\mathbf{1}_{A}(X)+G(X-\beta)\,\mathbf{1}_{A^c}(X)\bigr]
\]
over Borel sets $A\subset\mathbb{R}^d$ is given (up to $\mathbb{P}$-null sets) by
\[
A_{\alpha,\beta}:=\{x\in\mathbb{R}^d:\ G(x-\alpha)\le G(x-\beta)\}.
\]
Moreover, setting $\delta:=\beta-\alpha$, the set $A_{\alpha,\beta}$ has the following directional one-crossing
property: for every $z\in\mathbb{R}^d$, the intersection of $A_{\alpha,\beta}$ with the affine line
$z+\mathbb{R}\delta$ is either empty, all of $z+\mathbb{R}\delta$, or a half-line in that line.

In particular, for the quadratic loss $G(x)=\|x\|^2$ one has
\[
A_{\alpha,\beta}=\Big\{x\in\mathbb{R}^d:\ \langle x,\beta-\alpha\rangle\le \tfrac12\bigl(\|\beta\|^2-\|\alpha\|^2\bigr)\Big\},
\]
which is an affine halfspace.
\end{proposition}
\begin{proof}
\emph{Pointwise minimization (dimension-free).}
Fix $\alpha,\beta\in\mathbb{R}^d$ and minimize over Borel sets $A\subset\mathbb{R}^d$. Write
\[
\E\bigl[G(X-\alpha)\,\mathbf{1}_{A}(X)+G(X-\beta)\,\mathbf{1}_{A^c}(X)\bigr]
=\E\bigl[G(X-\beta)\bigr]+\E\bigl[(G(X-\alpha)-G(X-\beta))\,\mathbf{1}_{A}(X)\bigr].
\]
Thus, exactly as in one dimension, we minimize by choosing $\mathbf{1}_{A}(x)=1$ whenever
$G(x-\alpha)-G(x-\beta)\le 0$, i.e. by taking $A=A_{\alpha,\beta}$.

\medskip
\noindent\emph{Halfspace structure for $G(x)=\|x\|^2$.}
In one dimension, convexity of $G$ implies a monotonicity property of
$u\mapsto G(u+\delta)-G(u)$, which forces $A_{\alpha,\beta}$ to be a half-line.
In dimension $d\ge 2$ there is no analogous total order, so for a general convex $G$ the set
$A_{\alpha,\beta}=\{G(\cdot-\alpha)\le G(\cdot-\beta)\}$ need not be a halfspace.

For the quadratic loss, however,
\[
\|x-\alpha\|^2\le \|x-\beta\|^2
\quad\Longleftrightarrow\quad
2\langle x,\beta-\alpha\rangle\le \|\beta\|^2-\|\alpha\|^2,
\]
which is precisely the affine halfspace stated in the proposition.\qedhere
\end{proof}

\begin{example}\label{ex:convex-G-not-halfspace}
In general, for a convex $G$ the set $A_{\alpha,\beta}$ need not be a halfspace.
For instance, in $\mathbb{R}^2$ let
\[
G(x,y)=(x-y)^2+x^4,
\qquad
\alpha=(0,0),\ \beta=(1,0).
\]
Then $A_{\alpha,\beta}=\{(x,y):\ G(x,y)\le G(x-1,y)\}$ is given by
\[
(x-y)^2+x^4\le (x-1-y)^2+(x-1)^4
\quad\Longleftrightarrow\quad
y\ge 2x^3-3x^2+3x-1.
\]
Thus $A_{\alpha,\beta}$ is the epigraph of a cubic curve, hence it is not an affine halfspace.
\end{example}

We will work throughout with the quadratic loss
\begin{equation}\label{eq:Gquad}
G(x)=\|x\|^2,\qquad x\in\mathbb{R}^d,
\end{equation}
so that \eqref{eq:multi:general} becomes
\begin{equation}\label{eq:multi:main}
\underset{\alpha,\beta\in\mathbb{R}^d,\ A\in\mathcal{A}}{\mathrm{argmin}}\;
\mathbb{E}\Bigl[\bigl\|X-\alpha\mathbf{1}_{A}(X)-\beta\mathbf{1}_{A^c}(X)\bigr\|^2\Bigr].
\end{equation}

\medskip
\noindent\textbf{Optimizing first in $\alpha$ and $\beta$.}
Define
\[
h(\alpha,\beta,A):=\mathbb{E}\Bigl[\bigl\|X-\alpha\mathbf{1}_{A}(X)-\beta\mathbf{1}_{A^c}(X)\bigr\|^2\Bigr],
\qquad (\alpha,\beta)\in(\mathbb{R}^d)^2,\ A\in\mathcal{A}.
\]
The map $h$ is differentiable in $\alpha$ and $\beta$, and setting the gradients to zero yields
\begin{equation}\label{eq:multi:alpha}
\alpha=\frac{\mathbb{E}[X\mathbf{1}_A(X)]}{\mathbb{P}(A)}=\mathbb{E}[X\mid X\in A],
\qquad
\beta=\frac{\mathbb{E}[X\mathbf{1}_{A^c}(X)]}{\mathbb{P}(A^c)}=\mathbb{E}[X\mid X\notin A].
\end{equation}
Plugging \eqref{eq:multi:alpha} back into $h$ gives, after elementary rearrangements,
\begin{equation}\label{eq:multi:hA}
h(A)
=\mathbb{E}\|X\|^2
-\frac{\bigl\|\mathbb{E}[X\mathbf{1}_A(X)]\bigr\|^2}{\mathbb{P}(A)}
-\frac{\bigl\|\mathbb{E}[X\mathbf{1}_{A^c}(X)]\bigr\|^2}{\mathbb{P}(A^c)}.
\end{equation}
Therefore the minimization problem \eqref{eq:multi:main} is equivalent to the maximization problem
\begin{equation}\label{eq:multi:pbmax}
\underset{A\in\mathcal{A}}{\mathrm{argmax}}\;
\left(
\frac{\bigl\|\mathbb{E}[X\mathbf{1}_A(X)]\bigr\|^2}{\mathbb{P}(A)}
+\frac{\bigl\|\mathbb{E}[X\mathbf{1}_{A^c}(X)]\bigr\|^2}{\mathbb{P}(A^c)}
\right).
\end{equation}

\medskip
\noindent\textbf{Halfspaces as the natural class.}
Fix the mean $\mu:=\mathbb{E}[X]$. In view of Proposition~\ref{prop:multi-optimal-set}, for the quadratic loss the
regime boundary is an affine hyperplane, hence it is natural to restrict the maximization in
\eqref{eq:multi:pbmax} to halfspaces. We therefore consider the mean-centered halfspace class
\begin{equation}\label{eq:class-H}
\mathcal{H}_\mu
:=\Bigl\{A_{u,t}:\ u\in\mathbb{S}^{d-1},\ t\in\mathbb{R}\Bigr\},
\qquad
A_{u,t}:=\{x\in\mathbb{R}^d:\ \langle x-\mu,u\rangle\le t\}.
\end{equation}

\begin{remark}
The uncentered halfspace $\{x:\langle x,u\rangle\le \tau\}$ is the same as $A_{u,t}$ with
$\tau=\langle\mu,u\rangle+t$. Thus centering at $\mu$ is only a re-parameterization.
\end{remark}

If the sets $A$ are of the form $A=A_{u,t}$ we define the objective
\begin{equation}\label{eq:multi:Fut}
F(u,t)
:=\frac{\bigl\|\mathbb{E}[X\mathbf{1}_{A_{u,t}}(X)]\bigr\|^2}{\mathbb{P}(A_{u,t})}
+\frac{\bigl\|\mathbb{E}[X\mathbf{1}_{A_{u,t}^c}(X)]\bigr\|^2}{\mathbb{P}(A_{u,t}^c)},
\qquad \|u\|=1,\ t\in\mathbb{R}.
\end{equation}
In contrast with the one–dimensional case, even for weakly symmetric distributions it is \emph{not}
automatic that the maximizing halfspace is determined solely by the mean. The next subsections discuss
structural conditions (e.g.\ ellipticity) under which the optimizer can nevertheless be characterized.

The optimization problem in the multidimensional case is much more delicate and we will consider here only a special class of random variables for which we can argue about the identification of the maximum point as in the case of one dimensional case.  

Before we introduce some results we define the class of densities for which we can really prove some positive results about the identification of the mean as the optimal point for the separation space.

\begin{definition}[Elliptical Distribution]
    A random vector $X \in \mathbb{R}^d$ is said to have a centered \textit{elliptical distribution} with positive-definite scatter matrix $\Sigma \in \mathbb{R}^{d \times d}$ if its probability density function (PDF) exists and is of the form
    \[
        f_X(x) = C_d \det(\Sigma)^{-1/2} g\left( x^T \Sigma^{-1} x \right),
    \]
    where $g: [0, \infty) \to [0, \infty)$ is a non-negative density generator function satisfying $\int_0^\infty s^{d/2-1} g(s) \, ds < \infty$, and $C_d$ is the normalization constant.
\end{definition}

A convenient sufficient condition for log-concavity is that $g$ is log-concave and nonincreasing on $[0,\infty)$.
Equivalently, one may write $g(s)=e^{-\varphi(s)}$ with $\varphi:[0,\infty)\to\mathbb{R}$ convex and nondecreasing,
so that
\[
f_X(x)\propto \exp\!\left(-\varphi\!\left((x-\mu)^T\Sigma^{-1}(x-\mu)\right)\right)
\]
is log-concave. Examples are densities of the form $\exp(-\|\Sigma^{-1/2}(x-\mu)\|^p)$ with $p\ge 1$ which include the Gaussians as particular cases.

\begin{lemma}[Linearity of Conditional Expectation for Elliptical Distributions]
    Let $X \in \mathbb{R}^d$ be a centered random vector with an elliptical distribution and scatter matrix $\Sigma$. For any two linearly independent vectors $u, v \in \mathbb{R}^d$, the conditional expectation of the projection $\langle X, u \rangle$ given $\langle X, v \rangle = t$ is linear in $t$:
    \[
        \mathbb{E}[\langle X, u \rangle \mid \langle X, v \rangle = t] = \rho_{u,v} t,
    \]
    where $\rho_{u,v} = \frac{u^T \Sigma v}{v^T \Sigma v}$. Consequently, the function $\phi_{u,v}(t) = \mathbb{E}[\langle X, u \rangle \mid \langle X, v \rangle = t]$ satisfies $\phi(0)=0$ and the ratio $\phi(t)/t$ is constant (and thus non-increasing) on $(0, \infty)$.
\end{lemma}

\begin{proof}
    Consider the bivariate random vector $Y = (Y_1, Y_2)^T = (\langle X, v \rangle, \langle X, u \rangle)^T$. Since the class of elliptical distributions is closed under linear transformations, $Y$ follows a centered bivariate elliptical distribution with scatter matrix
    \[
        \Sigma_Y = \begin{pmatrix} v^T \Sigma v & v^T \Sigma u \ u^T \Sigma v & u^T \Sigma u \end{pmatrix} = \begin{pmatrix} \sigma_{11} & \sigma_{12} \ \sigma_{21} & \sigma_{22} \end{pmatrix}.
    \]
    The conditional density of $Y_2$ given $Y_1 = t$ is proportional to the slice of the joint density:
    \[
        f_{Y_2|Y_1}(y_2 \mid t) \propto g\left( \begin{pmatrix} t \ y_2 \end{pmatrix}^T \Sigma_Y^{-1} \begin{pmatrix} t \ y_2 \end{pmatrix} \right).
    \]
    Using the formula for the inverse of a block matrix, the quadratic form $Q(y_2) = (t, y_2) \Sigma_Y^{-1} (t, y_2)^T$ can be rewritten by completing the square with respect to $y_2$:
    \[
        Q(y_2) = c_1 \left( y_2 - \frac{\sigma_{12}}{\sigma_{11}} t \right)^2 + c_2(t),
    \]
    where $c_1 = (\det \Sigma_Y)^{-1} \sigma_{11} > 0$ and $c_2(t)$ is independent of $y_2$. 
    
    The conditional density is therefore of the form $h((y_2 - \mu(t))^2)$, where $\mu(t) = \frac{\sigma_{12}}{\sigma_{11}} t$. Since this density is symmetric about $\mu(t)$, the conditional expectation corresponds to the center of symmetry:
    \[
        \mathbb{E}[Y_2 \mid Y_1 = t] = \frac{\sigma_{12}}{\sigma_{11}} t = \frac{u^T \Sigma v}{v^T \Sigma v} t.
    \]
    This confirms the linearity of the regression function $\phi_{u,v}(t)$.
\end{proof}

\begin{definition}[Weak symmetry]\label{def:weak-sym}
A random vector $X\in\mathbb{R}^d$ with mean $\mu$ is called \emph{weakly symmetric} if for every halfspace $H$
whose boundary hyperplane contains $\mu$, one has $\mathbb{P}(X\in H)=1/2$.
Equivalently, for every unit $u\in\mathbb{S}^{d-1}$,
\[
\mathbb{P}(\langle X-\mu,u\rangle\le 0)=\frac12.
\]
\end{definition}

\begin{theorem}\label{t:dim:elli}
Let $X\in\mathbb{R}^d$ have an \emph{elliptical} density with mean $\mu$ and positive definite scatter matrix $\Sigma$.
Assume moreover that for every unit $u\in\mathbb{S}^{d-1}$, the one-dimensional projection
$Z_u:=\langle X-\mu,u\rangle$ has a \emph{weakly symmetric log-concave density} on $\mathbb{R}$
(e.g.\ this holds if $X$ has a log-concave density).
Then:

\smallskip
\noindent (i) For every fixed unit $u$, the map $t\mapsto F(u,t)$ is maximized at $t=0$.

\smallskip
\noindent (ii) The map $u\mapsto F(u,0)$ is maximized when $u$ is an eigenvector of $\Sigma$
corresponding to the largest eigenvalue $\lambda_{\max}(\Sigma)$.
\end{theorem}

\begin{proof}
\emph{Step 0: reduction to the centered case.}
Let $Y:=X-\mu$. For $A_{u,t}=\{\langle Y,u\rangle\le t\}$ we have
\[
\mathbb{E}[X\mathbf{1}_{A_{u,t}}]=\mu\,\mathbb{P}(A_{u,t})+\mathbb{E}[Y\mathbf{1}_{A_{u,t}}],
\qquad
\mathbb{E}[X\mathbf{1}_{A_{u,t}^c}]=\mu\,\mathbb{P}(A_{u,t}^c)+\mathbb{E}[Y\mathbf{1}_{A_{u,t}^c}].
\]
Since $\mathbb{E}[Y]=0$, we have $\mathbb{E}[Y\mathbf{1}_{A_{u,t}^c}]=-\mathbb{E}[Y\mathbf{1}_{A_{u,t}}]$,
and a direct expansion shows
\begin{equation}\label{eq:shift-decomp}
F(u,t)=\|\mu\|^2+\widetilde F(u,t),
\end{equation}
where
\[
\widetilde F(u,t)
:=\frac{\bigl\|\mathbb{E}[Y\,\mathbf{1}_{\{\langle Y,u\rangle\le t\}}]\bigr\|^2}{\mathbb{P}(\langle Y,u\rangle\le t)}
+\frac{\bigl\|\mathbb{E}[Y\,\mathbf{1}_{\{\langle Y,u\rangle> t\}}]\bigr\|^2}{\mathbb{P}(\langle Y,u\rangle> t)}.
\]
Hence maximizing $F(u,t)$ over $t$ is equivalent to maximizing $\widetilde F(u,t)$, so from now on we assume $\mu=0$.

\medskip
\emph{Step 1: reduce $\widetilde F(u,t)$ to a one-dimensional functional.}
Fix unit $u$ and set $Z:=\langle Y,u\rangle$. For centered elliptical $Y$ with scatter matrix $\Sigma$,
the linear regression property gives (take $v=u$ in the standard bivariate elliptical regression formula)
\begin{equation}\label{eq:ell-reg}
\mathbb{E}[Y\mid Z=s]=\frac{\Sigma u}{u^T\Sigma u}\,s.
\end{equation}
Therefore,
\[
\mathbb{E}\bigl[Y\,\mathbf{1}_{\{Z\le t\}}\bigr]
=\mathbb{E}\!\left[\mathbb{E}[Y\mid Z]\mathbf{1}_{\{Z\le t\}}\right]
=\frac{\Sigma u}{u^T\Sigma u}\,\mathbb{E}\bigl[Z\,\mathbf{1}_{\{Z\le t\}}\bigr].
\]
Since $\mathbb{E}[Y]=0$, we also have
$\mathbb{E}[Y\mathbf{1}_{\{Z>t\}}]=-\mathbb{E}[Y\mathbf{1}_{\{Z\le t\}}]$.
Hence
\begin{equation}\label{eq:F-factor}
\widetilde F(u,t)
=\frac{\|\Sigma u\|^2}{(u^T\Sigma u)^2}\,
\frac{\mathbb{E}[Z\,\mathbf{1}_{\{Z\le t\}}]^2}{\mathbb{P}(Z\le t)\mathbb{P}(Z>t)}.
\end{equation}
The $u$-dependent prefactor is constant in $t$, so maximizing $\widetilde F(u,t)$ over $t$ is equivalent to
maximizing the one-dimensional quantity
\[
t\ \longmapsto\ \frac{\mathbb{E}[Z\,\mathbf{1}_{\{Z\le t\}}]^2}{\mathbb{P}(Z\le t)\mathbb{P}(Z>t)}.
\]

\medskip
\emph{Step 2: maximize in $t$ (fixed $u$).}
By assumption, $Z$ has a weakly symmetric log-concave density on $\mathbb{R}$, with mean $0$.
Applying Theorem~\ref{t:1} (the one-dimensional monotonicity theorem) to $Z$ shows that the above
one-dimensional functional is maximized at $t=0$. By \eqref{eq:F-factor}, the same holds for $\widetilde F(u,t)$
and hence for $F(u,t)$. This proves (i).

\medskip
\emph{Step 3: maximize in $u$ at $t=0$.}
At $t=0$, weak symmetry gives $\mathbb{P}(Z\le 0)=\mathbb{P}(Z>0)=1/2$, and symmetry yields
$\mathbb{E}[Z\mathbf{1}_{\{Z\le 0\}}]=-\mathbb{E}[Z_+]$, where $Z_+=\max\{Z,0\}$.
Moreover, for elliptical $Y$ one has the scaling property
\[
Z=\langle Y,u\rangle \ \stackrel{d}{=}\ \sqrt{u^T\Sigma u}\,Z_0,
\]
where $Z_0$ is the projection in a fixed direction for the standardized scatter $\Sigma=I$ (its law does not depend on $u$).
Thus $\mathbb{E}[Z_+]=c_0\sqrt{u^T\Sigma u}$ for the constant $c_0:=\mathbb{E}[(Z_0)_+]>0$.
Plugging into \eqref{eq:F-factor} at $t=0$ gives
\[
\widetilde F(u,0)
=\frac{\|\Sigma u\|^2}{(u^T\Sigma u)^2}\cdot
\frac{\bigl(c_0^2(u^T\Sigma u)\bigr)}{(1/2)(1/2)}
=4c_0^2\,\frac{u^T\Sigma^2 u}{u^T\Sigma u}.
\]
Therefore maximizing $u\mapsto \widetilde F(u,0)$ over $u\neq 0$ is equivalent to maximizing
\[
R(u):=\frac{u^T\Sigma^2 u}{u^T\Sigma u}
=\frac{(\Sigma^{1/2}u)^T\Sigma(\Sigma^{1/2}u)}{\|\Sigma^{1/2}u\|^2},
\]
which is exactly the Rayleigh quotient of $\Sigma$ evaluated at the vector $\Sigma^{1/2}u$.
Hence $\sup_{u\neq 0}R(u)=\lambda_{\max}(\Sigma)$ and the maximizers are precisely the eigenvectors
associated to $\lambda_{\max}(\Sigma)$. This proves (ii). \qedhere
\end{proof}

\subsection{Considerations on convexity}\label{SS:3:3}\hfill\

In this section, we show that the convexity assumed in Theorem~\ref{t:1} is a key condition. To do so, we will proceed by providing an example of density of the form $e^{-V(x)}$ with $V(x)$ continuous and symmetric, but not convex and we will prove that for this case $c=0$ does not provide the maximum of the function $f$ defined in the Theorem \ref{t:1}.

\begin{example}[A symmetric non-log-concave law with two maximizers of $f_X$]
\label{ex:two-maxima-fx}
Let $X$ have the (even) density
\[
p(x)=
\begin{cases}
\frac{21}{8}, & |x|\le \frac{1}{10} \\
\frac{1}{8}, & \frac{1}{10}<|x|\le 2, \\
0, & |x|>2.
\end{cases}
\]
Then $p$ integrates to $1$ and, by symmetry, $\E[X]=0$.

\smallskip
\noindent\emph{Failure of log-concavity.}
For $x=0$, $y=0.3$ one has $p(0)=\frac{21}{8}$ and $p(0.3)=\frac18$, but
$p(0.15)=\frac18<\sqrt{p(0)p(0.3)}=\frac{\sqrt{21}}{8}$,
so $p$ is not log-concave.

\smallskip
\noindent\emph{Two global maximizers of $f_X$.}
For $t\in(-2,2)$ define
\[
f_X(t)=\frac{\E[X\mathbf 1_{\{X\le t\}}]^2}{\p(X\le t)}
+\frac{\E[X\mathbf 1_{\{X>t\}}]^2}{\p(X>t)}.
\]
We have $\E[X]=0$ and $f_X$ admits the explicit form
\begin{equation}\label{eq:fx-explicit-piecewise}
 f_X(t)=f_X(-t)=
 \begin{cases}
 \displaystyle
 \frac{\Bigl(-\frac{21}{80}+\frac{21}{16}t^2\Bigr)^2}{\Bigl(\frac12+\frac{21}{8}t\Bigr)\Bigl(\frac12-\frac{21}{8}t\Bigr)},
 & 0\le t\le \frac{1}{10},\\
 \displaystyle
 \frac{(2-t)(t+2)^2}{4(6+t)},
 & \frac{1}{10}\le t\le 2.
 \end{cases}
\end{equation}
By symmetry it suffices to consider $t\ge 0$.
Differentiating yields the following simplified expressions:
\[
 f_X'(t)=
 \begin{cases}
 \displaystyle
 \frac{441\,t\,(1-5t^2)(2205t^2+281)}{50\,(4-21t)^2(4+21t)^2},
 & 0<t<\frac{1}{10},\\
 \displaystyle
 -\frac{(t+2)(t^2+8t-4)}{2(6+t)^2},
 & \frac{1}{10}<t<2.
 \end{cases}
\]
In particular, on $(0,\frac{1}{10})$ one has $f_X'(t)>0$, and on $(\frac{1}{10},2)$ the only critical point is the root of
$t^2+8t-4=0$.
Therefore $f_X$ has a unique critical point on $(0,2)$ at
\[
t^\star=2\sqrt5-4\in\Bigl(\tfrac{1}{10},2\Bigr),
\]
and $f_X$ increases on $[\tfrac{1}{10},t^\star]$ and decreases on $[t^\star,2]$.
Comparing values shows $f_X(t^\star)>f_X(\tfrac{1}{10})$, hence the global maximizers are exactly $\pm t^\star$.
\end{example}

\begin{figure}[h]
\centering
\begin{tikzpicture}
\begin{axis}[
  width=0.46\textwidth,
  height=0.36\textwidth,
  grid=both,
  xlabel={$x$}, ylabel={$p(x)$},
  title={Density $p$}
]
\addplot[domain=-2.2:2.2, samples=600]
{ifthenelse(abs(x)<=0.1, 21/8, ifthenelse(abs(x)<=2, 1/8, 0))};
\end{axis}
\end{tikzpicture}
\hfill
\begin{tikzpicture}
\begin{axis}[
  width=0.46\textwidth,
  height=0.36\textwidth,
  grid=both,
  xlabel={$t$}, ylabel={$f_X(t)$},
  title={$f_X$ (two global maxima)},
  ymin=0
]
\addplot[domain=-1.95:1.95, samples=900]
{%
  (abs(x)<=0.1) *
    ( ((-21/80 + (21/16)*abs(x)^2)^2) /
      ((1/2+(21/8)*abs(x))*(1/2-(21/8)*abs(x))) )
 + (abs(x)>0.1 && abs(x)<2) *
    ( ((2-abs(x))*(abs(x)+2)^2)/(4*(6+abs(x))) )
};
\addplot[gray,dashed] coordinates {(0.47213595,0) (0.47213595,0.38)};
\addplot[gray,dashed] coordinates {(-0.47213595,0) (-0.47213595,0.38)};
\end{axis}
\end{tikzpicture}
\caption{Example~\ref{ex:two-maxima-fx}: a symmetric, non-log-concave density with global maximizers of $f_X$
at $\pm t^\star$, where $t^\star=2\sqrt5-4\approx 0.4721$.}
\end{figure}
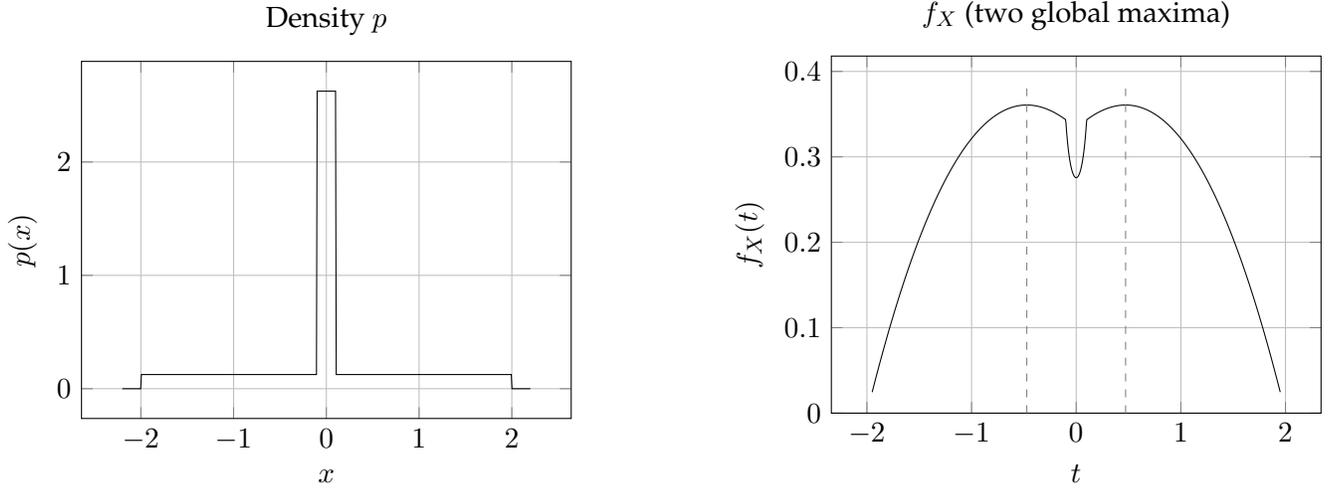

\subsection{Counterexample for the halfspace functional (integer vertices)}\label{subsec:halfspace-ctrex}

Let $X=(X_1,X_2)\sim \mathrm{Unif}(K)$ where $K\subset\mathbb{R}^2$ is the convex hexagon with
\emph{integer} vertices (listed counterclockwise)
\[
A=(-3,0),\quad B=(-1,-12),\quad C=(3,-8),\quad D=(3,0),\quad E=(1,12),\quad F=(-3,8).
\]
This hexagon is visibly non-degenerate: the vertical edges $FA$ and $CD$ have lengths $8$ and $8$.

Define
\[
R(t):=\frac{\big\|\mathbb{E}\big[X\,\mathbf{1}_{\{X_1>t\}}\big]\big\|^2}{\mathbb{P}(X_1>t)\,\mathbb{P}(X_1<t)}.
\]
For a uniform law on $K$, constants cancel and
\[
R(t)=\frac{\left\|\iint_{K\cap\{x>t\}}(x,y)\,dA\right\|^2}
{\Area(K\cap\{x>t\})\,\Area(K\cap\{x<t\})}.
\]

\medskip
\noindent\textbf{Centering.}
A shoelace/Green computation gives
\[
\Area(K)=104,\qquad \iint_K x\,dA=0,\qquad \iint_K y\,dA=0,
\]
hence $\mathbb{E}[X]=(0,0)$.

\medskip
\noindent\textbf{The cut $t=0$.}
The line $x=0$ intersects $BC$ at $(0,-11)$ and $EF$ at $(0,11)$, so
\[
K\cap\{x>0\}=\mathrm{conv}\{(0,-11),\,C,\,D,\,E,\,(0,11)\}=:P_0.
\]
One finds
\[
\Area(P_0)=52,\qquad \iint_{P_0}(x,y)\,dA=\left(\frac{199}{3},-\frac{67}{3}\right),
\qquad \Area(K\cap\{x<0\})=52,
\]
hence
\[
R(0)=\frac{\left(\frac{199}{3}\right)^2+\left(\frac{67}{3}\right)^2}{52\cdot 52}
=\frac{22045}{12168}\approx 1.8117.
\]

\medskip
\noindent\textbf{The cut $t=1$.}
The line $x=1$ intersects $BC$ at $(1,-10)$ and meets $E=(1,12)$, so
\[
K\cap\{x>1\}=\mathrm{conv}\{(1,-10),\,C,\,D,\,(1,12)\}=:P_1.
\]
One finds
\[
\Area(P_1)=30,\qquad \iint_{P_1}(x,y)\,dA=\left(\frac{166}{3},-\frac{100}{3}\right),
\qquad \Area(K\cap\{x<1\})=74,
\]
hence
\[
R(1)=\frac{\left(\frac{166}{3}\right)^2+\left(\frac{100}{3}\right)^2}{30\cdot 74}
=\frac{9389}{4995}\approx 1.8797.
\]
Therefore $R(1)>R(0)$, so even for a centered law ($\mathbb{E}X=0$) the halfspace functional in direction $e_1$
need not be maximized at the centered cut $t=0$ (nor decreasing for $t\ge 0$) without additional structure.

\begin{figure}[h]
\centering
\begin{tikzpicture}[scale=0.25]
  \coordinate (A) at (-3,0);
  \coordinate (B) at (-1,-12);
  \coordinate (C) at (3,-8);
  \coordinate (D) at (3,0);
  \coordinate (E) at (1,12);
  \coordinate (F) at (-3,8);

  \fill[orange!18] (B)--(1,-10)--(1,12)--(E)--(F)--(A)--cycle;
  \fill[blue!18] (1,-10)--(C)--(D)--(1,12)--cycle;

  \draw[thick] (A)--(B)--(C)--(D)--(E)--(F)--cycle;

  \draw[dashed,thick] (1,-14)--(1,14) node[right] {$x=1$};
  \draw[dotted,thick] (0,-14)--(0,14) node[left] {$x=0$};

  \draw[->] (-6,0)--(6,0) node[right] {$x$};
  \draw[->] (0,-14)--(0,14) node[above] {$y$};
\end{tikzpicture}
\caption{Integer-vertex hexagon $K$ and slices $x=0$ (dotted), $x=1$ (dashed).}
\end{figure}
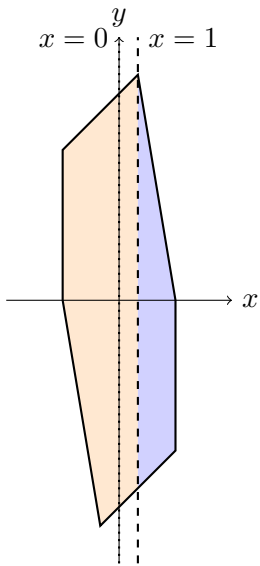

Since $R(1)>R(0)$, the halfspace objective in direction $e_1$ need not be maximized at the centered cut $t=0$ which does not agree with the case of elliptic case discussed in Theorem~\ref{t:dim:elli}.

\begin{remark}
Throughout the paper we primarily worked with laws admitting a density on the whole space $\mathbb{R}^d$
(e.g.\ log-concave densities of the form $e^{-V}$).
Nevertheless, in our counterexamples and geometric computations we consider uniform laws on bounded
convex sets. These can be viewed as densities supported on a convex body (with respect to Lebesgue measure), and the
associated halfspace functionals are still well-defined.

Moreover, such bounded-support examples can be approximated by fully supported laws: for instance, one may convolve
$\mathrm{Unif}(K)$ with a small centered Gaussian noise $\mathcal{N}(0,\varepsilon^2 I)$.
Then the resulting law has a smooth, everywhere positive density on $\mathbb{R}^d$ and its associated halfspace
functionals converge (as $\varepsilon\downarrow 0$) to those of $\mathrm{Unif}(K)$.
In particular, any counterexample exhibited on a bounded convex domain yields (by approximation) counterexamples
within the class of fully supported densities.
\end{remark}

\bibliographystyle{plainurl}
\bibliography{Risk28_final
}

\end{document}